\documentclass{article}[14pt]
\usepackage{amsmath, epsfig,amssymb, multicol}
\usepackage{amsthm}
\usepackage[active]{srcltx}
\usepackage{epstopdf}
\usepackage{graphicx}
\usepackage{color}
\usepackage{setspace}
\usepackage[colorlinks=true,linkcolor=blue,citecolor=blue,urlcolor=blue]{hyperref}
\usepackage[nameinlink,noabbrev]{cleveref}

\theoremstyle{plain}
\newtheorem{theorem}{Theorem}
\newtheorem{conjecture}{Conjecture}
\newtheorem{lemma}{Lemma}

\newtheorem{claim}{Claim}

\newtheorem{fact}{Fact}
\newcommand{\ex}{\operatorname{ex}}

\newcommand{\cC}{\mathcal C}
\newcommand{\cH}{\mathcal H}
\newcommand{\cM}{\mathcal M}

\newcommand{\dG}{d_G}
\newcommand{\Nbr}{N_G}
\newcommand{\codeg}{\operatorname{codeg}}

\title{Counting triangles in graphs without  suspensions of  paths}
\author{
Chunyang Dou\thanks{School of Mathematical Sciences, Anhui University, Hefei 230601, P. R. China. Email:
\texttt{chunyang@stu.ahu.edu.cn}. Supported in part by the National Natural Science Foundation of China (No. 12471319) and Anhui
Provincial Department of Education Research Project (No.~2025AHGXZK50154).}
\and Xing Peng\thanks{Center for Pure Mathematics, School of Mathematical Sciences, Anhui University, Hefei 230601, P. R. China.
Email: \texttt{x2peng@ahu.edu.cn}. Supported in part by the National Natural Science Foundation of China (No. 12471319) and
Excellent University Research and Innovation Team in Anhui Province (No. 2025AHGXZK10041).}
}
\date{}

\begin{document}

\maketitle

\begin{abstract}
For a graph $H$,  the suspension $\widehat H$ of $H$  is a graph  obtained from $H$
by joining a new vertex  to every vertex of $H$. The generalized Tur\'an number
$\ex(n,K_3,\widehat H)$ is the maximum number of triangles in an
$n$-vertex $\widehat H$-free graph. For fixed $k \ge 4$, Mubayi and Mukherjee conjectured that
\[
  \ex(n,K_3,\widehat P_k)
  =\left\lfloor\frac{k-2}{2}\right\rfloor\frac{n^2}{8}+o(n^2),
\]
as $n \to \infty$,
where $P_k$ is the path on $k$ vertices. They verified the conjecture for $k=4,5,6$.
In this paper, we determine the asymptotic value of $\ex(n,K_3,\widehat P_k)$ for every $k\ge 4$ via a unified approach and thus prove the conjecture of Mubayi and Mukherjee.

\end{abstract}

\section{Introduction}\label{sec:introduction}
The study of the Tur\'an number of graphs is one of the central themes of extremal graph theory.
Let $F$ be a graph. If $G$ does not contain $F$ as a subgraph, then we say $G$ is $F$-free. The {\it Tur\'an number}
   $\ex(n, F)$ is the maximum number of edges in an $n$-vertex $F$-free graph. An $F$-free graph with $\ex(n,F)$ edges is called an {\it extremal graph}.
   The first result in this area  is perhaps Mantel's Theorem for $\ex(n, K_3)$. The celebrated Tur\'an theorem \cite{Turan} gives $\ex(n, K_{t})$ for all $t \geq 3$. Erd\H{o}s \cite{E62} extended Tur\'an's theorem and determined the maximum number of $K_s$ in an $n$-vertex $K_t$-free graph with $s<t$. In general, for two graphs $F$ and $T$, the {\it generalized Tur\'an number} $\ex(n,T,F)$ equals the maximum number of copies of $T$ in an $n$-vertex $F$-free graph. Following Erd\H{o}s' result on $\ex(n,K_s,K_t)$,  there were only limited studies addressing the generalized Tur\'an number for a few decades. Because of Alon and Shikhelman's systematic study \cite{AS16} of $\ex(n,T,F)$, this topic has recently attracted a great deal of attention. For a detailed introduction to this topic, one can refer to the survey \cite{GP26} by Gerbner and  Palmer.

The study of $\ex(n,K_3,F)$ is of particular interest. For a graph $G$, we use $t(G)$ to denote the number of triangles in $G$. Given a vertex $v\in V(G)$, let $\Nbr(v)=\{u:u\text{ is
adjacent to }v\}$ be the {\it neighborhood} of $v$ and let $\dG(v)=|\Nbr(v)|$ be its {\it degree}. One can easily observe the following identity
\begin{equation*}
 t(G)=\frac13\sum_{v\in V(G)} e(G[\Nbr(v)]),
\end{equation*}
where $e(G[\Nbr(v)])$ is the number of edges in the subgraph induced by $\Nbr(v)$. To get an upper bound for $t(G)$, one can assume that the subgraph induced by the neighborhood of each vertex is $H$-free for a given $H$.
If we define the graph obtained from $H$ by joining a new vertex to every vertex of $H$ as the suspension  $\widehat H$ of $H$, then $G$ being  $\widehat H$-free is equivalent to $G[N_G(v)]$ being $H$-free for each vertex $v$. This idea connects the classical Tur\'an problem to the triangle-counting problem.  Moreover,
\[
 \ex(n,K_3,\widehat H)
 \le \frac13\sum_{v\in V(\Gamma)} \ex(d_{\Gamma}(v),H),
\]
 where $\Gamma$ is an extremal graph for $\ex(n,K_3,\widehat{H})$.

Mubayi and Mukherjee~\cite{MM23} initiated the systematic study of
$\ex(n,K_3,\widehat H)$ for  bipartite $H$.
Let $P_k$ be the path on $k$ vertices.  Mubayi and
Mukherjee~\cite{MM23}  proposed the following  conjecture.
\begin{conjecture} \label{MMC}
  For every fixed $k\ge 4$, as $n \to \infty$,
  \begin{equation*}
      \ex(n,K_3,\widehat P_k)
  =\left\lfloor\frac{k-2}{2}\right\rfloor\frac{n^2}{8}+o(n^2).
  \end{equation*}
\end{conjecture}
They verified this conjecture for $k=4,5,6$. In addition, Mukherjee \cite{M24} determined the exact value of $ \ex(n,K_3,\widehat P_4)$ for all $n$. Note that a similar result was proved by Gerbner \cite{G22} for large $n$.  Using a stability theorem for the Tur\'an problem of a non-bipartite graph, Hei, Hou, and Ma \cite{HHM} recently determined
the exact value of $\ex(n,K_3,\widehat P_5)$ for sufficiently large $n$.

In this paper, we determine the asymptotic value of $\ex(n,K_3,\widehat P_k)$ for every $k\ge 4$ using a unified method and thus prove  Conjecture \ref{MMC}.
\begin{theorem}\label{thm:main}
For every fixed integer $k\ge 4$, as  $n \to \infty$,
\[
  \ex(n,K_3,\widehat P_k)
  =
  \left\lfloor\frac{k-2}{2}\right\rfloor\frac{n^2}{8}+o(n^2).
\]
\end{theorem}

There are a number of related results in the same spirit.  Note that the fan $F_k$ is the suspension of $M_k$, where $M_k$ is the matching of size $k$.  Erd\H{o}s
and S\'os \cite{ESos} established the exact value of
$\ex(n,K_3, F_2)$.  Zhu, Chen, Gerbner, Gy\H{o}ri, and Hama Karim \cite{Zhu} obtained the exact value of $\ex(n,K_3, F_k)$ for each $k$ and large $n$.  The suspension of a $k$-cycle $C_k$ is the wheel $W_k$ on $k+1$ vertices. There are several papers \cite{DNP26,FGK,Zhou} studying $\ex(n,K_3,W_{\geq k})$ for $k=3,4$, where $W_{\geq k}$ is the set of wheels with at least $k+1$ vertices.

Throughout the remainder of the paper, $k\ge4$ is fixed and
\[
  r:=\left\lfloor\frac{k-2}{2}\right\rfloor.
\]
Thus $k\in\{2r+2,2r+3\}$ and, in particular, $2r+2\ge k-1$. All
implicit constants depend on $k$.

The proof of Theorem \ref{thm:main} relies on the following auxiliary result which may be of independent interest.

\begin{theorem}\label{lem:weighted}
If $G$ is a $P_k$-free graph with a  partition $V(G)=X\cup Y$ such that $X$ is
an independent set  and $d_X(y)\le r$ as well as $d_{Y}(y)\le r$ for every $y\in Y$,
then
\begin{equation}\label{eq:weighted-conclusion}
  e_G(X,Y)+2e_G(Y)\le r|Y|+C_k|X|,
\end{equation}
where $C_k$ is a constant depending on $k$.
\end{theorem}
The following  well-known Erd\H{o}s--Gallai theorem will be frequently used in the proof of our main result.
\begin{theorem}[Erd\H{o}s--Gallai~\cite{EG}]\label{lem:EG}
If a graph $G$ does not contain  $P_k$ for $k \geq 2$, then
\[
  e(G)\le \frac{k-2}{2}|V(G)|.
\]
\end{theorem}

{\bf Notation:} All graphs are finite, simple, and undirected. For a graph $G$, we write
$e(G)$ and $t(G)$ for its numbers of edges and triangles, respectively.
For disjoint sets $X,Y\subseteq V(G)$, let $e_G(X,Y)$ denote the number of edges with one
endpoint in $X$ and the other in $Y$. For $X\subseteq V(G)$, we write $e_G(X)=e(G[X])$, where $e(G[X])$ is the number of edges in the subgraph induced by $X$. Additionally, we define $N_X(v)=N_G(v) \cap X$ and $d_X(v)=|N_X(v)|$ for a vertex $v$.
 The
subscript is omitted when the  graph is clear from the context. For two distinct vertices $u$ and $v$, their \emph{codegree} is $ \codeg_G(uv):=|N_G(u)\cap N_G(v)|.$

The rest of the paper is organized as follows. We will prove  Theorem \ref{lem:weighted} in Section \ref{sec:weighted-proof}.
Theorem \ref{thm:main} will be established in
Section \ref{sec:main-proof}.

\section{Proof of Theorem \ref{lem:weighted}}\label{sec:weighted-proof}

For a connected graph $G$ and a vertex $v\in V(G)$, define $\ell_G(v)$ as the length of  a longest path of $G$ starting at $v$. We will use the following variant of the Erd\H{o}s--Gallai theorem.

\begin{lemma}\label{lem:rooted-EG}
Let $G$ be a connected graph on $n \ge2$ vertices and let
$v\in V(G)$. Then
\[
  \ell_G(v)\ge \frac{2e(G)-d_G(v)}{n-1}.
\]
Equivalently, $2e(G)-d_G(v)\le \ell_G(v)(n-1)$.
\end{lemma}

\begin{proof}
We prove the lemma by induction on $n$. The base case where $n=2$ is trivial.

Assume that $n\ge3$ and  the assertion holds for every connected
graph with fewer than $n$ vertices. There are two cases.

\smallskip
\noindent\emph{Case 1: $G-v$ is disconnected.}
Let $C_1,\dots,C_q$ be the connected components of $G-v$ with $q\ge2$. For each $i\in\{1,\dots,q\}$, set
\[
  m_i:=|C_i|,\qquad
  e_i:=e(G[C_i]),\qquad
  d_i:=d_{C_i}(v).
\]
Observe that $d_i\ge1$ for every $i$ and
$G_i:=G[C_i\cup\{v\}]$ is connected with $2 \leq m_i+1<n$ vertices. The
induction hypothesis applied to $G_i$  gives
\[
  \ell_{G_i}(v)
  \ge
  \frac{2e(G_i)-d_{G_i}(v)}{|V(G_i)|-1}
  =\frac{2(e_i+d_i)-d_i}{m_i}
  =\frac{2e_i+d_i}{m_i}.
\]
Every path in $G_i$ beginning with $v$ is also a path in $G$ beginning with
$v$. Therefore
\begin{align*}
  \ell_G(v)\ge \max_{1\le i\le q}\ell_{G_i}(v)\ge \max_{1\le i\le q}\frac{2e_i+d_i}{m_i}\ge \frac{\sum_{i=1}^q(2e_i+d_i)}{\sum_{i=1}^q m_i}.
\end{align*}
The last inequality is true because  the maximum of finitely many ratios is no less than their weighted average.
 Finally,
\[
  \sum_{i=1}^q m_i=n-1\qquad\text{and}\qquad\sum_{i=1}^q(2e_i+d_i)=2e(G)-d_G(v),
\]
because $e(G)=\sum_{i=1}^q (e_i+d_i)$ and $d_G(v)=\sum_{i=1}^q d_i$. This proves the
desired inequality for this case.

\smallskip
\noindent\emph{Case 2: $H:=G-v$ is connected.}
Let $h:=|V(H)|=n-1$. It is trivial that  $h \ge2$ and $d_G(v) \ge1$. Define
\[
  a:=\sum_{u \in N_G(v)} d_H(u).
\]
On the one hand,  $a \le 2e(H)$. On the other
hand, every vertex of $H$ has degree at most $h-1$, so
$a\le d_G(v)(h-1)$. Let
\[
  b:=\frac{h(h-1)+2e(H)-(h-1)d_G(v)}{h}.
\]
We claim that
\begin{equation}\label{eq:B-min-bound}
  \min\left\{h-1,\frac{2e(H)}{d_G(v)}\right\}\le b.
\end{equation}
If $2e(H) \ge d_G(v)(h-1)$, then
\[
  b-(h-1)=\frac{2e(H)-d_G(v)(h-1)}{h}\ge 0.
\]
If $2e(H)\le d_G(v)(h-1)$, then, noting $d_G(v) \le h$,
\[
  b-\frac{2e(H)}{d_G(v)}
  =\frac{(h-d_G(v))\bigl(d_G(v)(h-1)-2e(H)\bigr)}{d_G(v) h}\ge0.
\]
This proves the inequality \eqref{eq:B-min-bound}.

By averaging, there is a vertex  $u \in N_G(v)$ with $d_H(u)\le \tfrac {a}{d_G(v)}
  \le\min\left\{h-1,\tfrac{2e(H)}{d_G(v)}\right\}
  \le b$ as we showed $a \leq 2e(H)$ and $a \leq d_G(v)(h-1)$ above.
By the induction hypothesis, a longest path in $H$ beginning with $u$ together with the edge $vu$
forms a  path in $G$ beginning with $v$.
Consequently,
\begin{align*}
  \ell_G(v)\ge 1+\ell_H(u)\ge 1+\frac{2e(H)-d_H(u)}{h-1}\ge 1+\frac{2e(H)-b}{h-1}=\frac{2e(H)+d_G(v)}{h}=\frac{2e(G)-d_G(v)}{n-1}.
\end{align*}
The second-to-last equality can be verified by plugging in the definition of $b$.  For the last equality, we use $e(G)=e(H)+d_G(v)$. This completes the induction.
\end{proof}

Next, we introduce a few more definitions.
 A {\it multi-hypergraph} is a pair $\cH=(V,\mathcal E)$, where $V$ is the vertex set and $\cal E$ is
 a multiset of nonempty subsets of $V$.  The {\it multiplicity} $m(e)$ of an edge $e \in \cal E$ is the number of occurrences of $e$ in $\cal E$ and the multiplicity of $\cH$ equals $\max\{m(e): e \in \cal E\}$.
A \emph{Berge path} of length $q$ in a multi-hypergraph $\cH=(V,\mathcal E)$ is an alternating
sequence $v_0,e_1,v_1,e_2,\dots,e_q,v_q$,
where $v_0,\dots,v_q \in V$ are distinct vertices
and $v_{i-1},v_i\in e_i$ for every $i\in\{1,\dots,q\}$. In addition, the number of occurrences of $e_i$ in the sequence is at most $m(e_i)$.

\begin{lemma}\label{lem:hypergraph}
Let $\cH=(V, {\cal E})$ be a multi-hypergraph with
multiplicity at most $\mu$. If $\cH$ contains no Berge path of length
$r+1$, then
\begin{equation*}
  e(\cH)\le\bigl(r+\mu2^{r+1}\bigr)|V(\cH)|.
\end{equation*}
\end{lemma}

\begin{proof} We prove the lemma by induction on $|V(H)|$. The case of $|V(H)|=1$ is trivial. For the induction step,
assume  $P=v_0,e_1,v_1,e_2,\dots,e_q,v_q$ is a longest Berge path in $\cH$.
  Each hyperedge copy containing $v_0$ that is not used by $P$ must have all its vertices in $\{v_0,\dots,v_q\}$. Otherwise, we can get a  Berge path of length $q+1$ by adding this hyperedge copy.  Note that there are at most $2^{q+1}$  subsets of $\{v_0,\dots,v_q\}$ and the  multiplicity  of $\cH$ is at most $\mu$. Thus the number of  edges containing $v_0$
is at most $q+\mu 2^{q+1} \leq r+ \mu 2^{r+1}$. Remove the vertex $v_0$ together with all edges containing $v_0$.
Let $\cH'$ be the resulting multi-hypergraph. Note that $\cH'$ is a multi-hypergraph with $|V(\cH)|-1$ vertices and also does not contain a Berge path of length $r+1$. Additionally, the multiplicity of $\cH'$ is at most $\mu$. By induction, $e(\cH') \leq \bigl(r+\mu2^{r+1}\bigr)|V(\cH')|=\bigl(r+\mu2^{r+1}\bigr)(|V(\cH)|-1).$ Together with the upper bound for the number of edges containing $v_0$, we get $e(\cH) \leq  \bigl(r+\mu2^{r+1}\bigr)|V(\cH)|$.
\end{proof}

 Let $Z$ be a graph with a partition $V(Z)=L \cup R$.
For an integer $m\ge1$, the \emph{$m$-fold bouquet}
$\mathcal B_m(Z,L,R)$ is a graph with vertex set $\cup_{i=1}^m L_i \cup R$ such that the subgraph induced by $L_i \cup R$ is isomorphic to $Z$ for each $1 \leq i \leq m$. The next lemma is a key ingredient in the proof of Theorem \ref{lem:weighted}.
\begin{lemma}\label{lem:bouquet}
Let $Z$ be a graph with a partition $V(Z)=L \cup R$. Assume that $Z[L]$ is connected and $R$ is an independent set.  In addition,   $d_L(v) \geq 1$ for each $v \in R$.   If  $d_L(v)\le r$ for each $v \in L$  and
$d_R(v)\le r$ for every $v \in L$  and
\begin{equation}\label{eq:positive-excess}
  e(L,R)+2e(L)>r|L|,
\end{equation}
then $\mathcal B_{s+1}(Z,L,R)$ contains a path with at least $2r+2$
edges, where $s=|R|$. In particular, it contains a copy of $P_k$.
\end{lemma}
%We postpone the proof of this lemma until the end of this section and proceed to prove Theorem \ref{lem:weighted}.

\begin{proof} For each $v \in L$, we define the deficit of $v$ as  $\eta(v):=r-d_L(v)$
 and put $D:=\sum_{v\in L}\eta(v)=r|L|-2e(L).$
Condition \eqref{eq:positive-excess} is therefore equivalent to $e(L,R)>D.$

We first show $|L| \geq 2$. If  $L$ contains only one vertex, then $e(L)=0$ and   $e(L,R)+2e(L) \leq r|L|$ as
 $e(L,R)\le r|L|$ by the assumption. This is    a contradiction to the assumption.

Recall  $L_i$ is connected and  $d_{L_i}(v) \geq 1$ for each  $v \in R$ and $1 \leq i \leq s+1$. Therefore, for any distinct $u,v \in R$ and any $L_i$ for $1 \leq i \leq s+1$, there is a $uv$-path with length at least two and all internal vertices in $L_i$. We have two cases.

If $s\ge r+1$, then we write  $R=\{v_1,\dots,v_s\}$. For each
$i\in\{2,\dots,s\}$, let $P_i$ be a $v_{i-1}v_{i}$-path with internal vertices in $L_i$. Let $P_1=v_1u_1$ and $P_{s+1}=v_su_{s+1}$, where $u_1 \in N_{L_1}(v_1)$ and   $u_{s+1} \in N_{L_{s+1}}(v_s)$. Putting  all paths together, we get a path with length at least $2s \ge 2r+2 \geq k-1$, as desired.

It remains to consider the case where $1 \leq s \leq r$. For  $v \in L$, let $\ell_L(v)$ be the length of a longest path in $L$ which  begins with $v$. Applying Lemma
\ref{lem:rooted-EG}, we get
\begin{align*}
  \ell_L(v)(|L|-1) &\ge 2e(L)-d_L(v)\\
  &=r|L|-D-\bigl(r-\eta(v)\bigr)\\
  &=r(|L|-1)-D+\eta(v).
  \label{eq:rooted-deficit}
\end{align*}
Call $v \in L$ \emph{long} if  $\eta(v)>D-s(|L|-1)$.
For every long  vertex $v \in L$, we have
\[
  \ell_L(v)(|L|-1)
  > (r-s)(|L|-1).
\]
As $|L| \geq 2$ and $\ell_L(v)$ is an integer, it follows  that $\ell_L(v)\ge r-s+1.$

\begin{claim}\label{clm:long-attachments}
If $s=1$, then the unique vertex in $R$ is adjacent to a long vertex in $L$. If $2\le s\le r$, then there are two distinct vertices in $R$ such that each of them is adjacent to a long vertex in $L$.
\end{claim}

\begin{proof}[Proof of Claim \ref{clm:long-attachments}]
If $D<s(|L|-1)$, then each $v \in L$ is long by the definition and   the claim follows
immediately.

Next we assume $D\ge s(|L|-1).$ Apparently,  $e(L,R)\le s|L|$. Note that the assumption $e(L,R)+2e(L)>r|L|$ implies $D<e(L,R)$. We may write
\[ D=s|L|-t \mbox{ for some }  t\in\{1,\dots,s\}.\]
Let  $M:=s|L|-e(L,R)$  be the number of missing edges between $L$ and $R$. Since
$e(L,R)>D=s|L|-t$, we obtain $$M\le t-1.$$ Using the parametrization of $D$, a vertex $v \in L$ is long if   $\eta(v)\ge s-t+1.$

We assert that there is at least one long vertex. Otherwise,
$\eta(v)\le s-t$ for every $v \in L$, which implies that  $D=\sum_v \eta(v)\le |L|(s-t).$ This is a contradiction to the assumption $D= s|L|-t$, as $t\ge 1$ and $|L|\ge 2$.
Note that there are at most $M$ edges missing between $L$ and $R$. Thus  each long vertex in $L$ has at least
$ s-M=|R|-M \ge s-t+1$ neighbors in $R$.

 If $s=1$, then  $t=1$ as $t\in\{1,\dots,s\}$ and any long vertex is adjacent to the
unique vertex in $R$. This proves the first part of the claim.

Now assume that $2\le s\le r$. If $t\le s-1$, then
$s-t+1\ge2$.  Each long vertex is adjacent to at least two distinct
vertices in $R$ and the second part of the claim follows.

It remains to prove the case of $t=s$. In this case,  a vertex  $v \in L$ is long if
$$\eta(v)>D-s(|L|-1)=s|L|-t-s(|L|-1)=0,$$
i.e., $\eta(v) \geq 1$. Let $L' \subseteq L$ be the set of long vertices. If $|N_R(L')| \geq 2$, then we are done.
Suppose $|N_R(L')|=1$.
We next show that $s\le 1$, contradicting the assumption $s\ge 2$. This will prove the second part of the claim.
  As we assume $|N_R(L')|=1$, each long vertex contributes  at least $s-1$ missing edges between $L$ and $R$.  Additionally, the number of missing edges is at most $M\le t-1=s-1$. Thus there is exactly one long vertex, say $v$, and
$M=s-1$. As  a vertex is long if $\eta(v)=r-d_L(v) \ge 1$ in this case, each vertex $u \in L$ with $u \neq v$ satisfies $d_L(u)=r \leq |L|-1$. Notice that $\eta(v)=D=s|L|-t=s(|L|-1)$ as we assume $v$ is the unique vertex with $\eta(v)>0$ and $s=t$. It is clear that $\eta(v) \leq r \leq |L|-1$.  We  obtain
\[
  s(|L|-1)\le r\le |L|-1,
\]
which forces $s\le1$, a contradiction. The claim is proved.
\end{proof}

We continue to show the existence of a $P_k$ for the case $1 \leq s \leq r$. For $s=1$, let $u$ be the unique vertex in $R$. By Claim \ref{clm:long-attachments},  $u$ is adjacent to a long  vertex in $L$, say $v$. Let $P$ be a longest path in $L$ starting with $v$. Note that $P$ has at least $r$ edges. Let $P_1$ and $P_2$ be copies of $P$ in $L_1$ and $L_2$. By the definition of   $\mathcal B_{s+1}(Z,L,R)$, it contains a path with length at least $2r+2 \geq k-1$, as desired.

For $s\ge 2$, write $R=\{u_1,\ldots,u_s\}$. By Claim \ref{clm:long-attachments}, after relabeling the vertices of $R$, we may assume that $u_1$ and $u_s$ are adjacent to long vertices $v_1$ and $v_{s+1}$ in $L$, respectively. These two long vertices need not be distinct. Let $P_1$ be a longest path in $L_1$ starting at the copy of $v_1$, and let $P_{s+1}$ be a longest path in $L_{s+1}$ starting at the copy of $v_{s+1}$. Each of these paths has length at least $r-s+1$. For each $2\le i\le s $, choose a $u_{i-1}u_i$-path $P_i$ with all internal vertices in $L_i$. Concatenating these paths with the attachment edges at $u_1$ and $u_s$  gives a path of length at least
$$ 2(r-s+2)+2(s-1)=2r+2. $$
This completes the proof.
\end{proof}
We are now ready to prove Theorem \ref{lem:weighted}.
\begin{proof}[Proof of Theorem \ref{lem:weighted}]
Let $\cC$ be the set of connected components of $G[Y]$. By the assumption, every
$C\in\cC$ has maximum degree at most $r$. Moreover, the diameter of $C$
is at most $k-2$ as $G$ is $P_k$-free. The Moore bound gives  $|C|\le M_k:=1+r+r^2+\cdots+r^{k-2}.$
For each $C\in\cC$, set
\[
  n_C:=|C|,\qquad e_C:=e(C),\qquad z_C:=e_G(C,X),
\]
and define the \emph{excess} of $C$ as
\[
  \gamma(C):=z_C+2e_C-rn_C.
\]
Observe that $e_G(X,Y)+2e_G(Y)-r|Y|=\sum_{C\in\cC}\gamma(C).$
Apparently,  $z_C\le rn_C$
and  $2e_C \le r n_C$ as $d_X(y)\le r$ and $d_{Y}(y)\le r$ for every $y\in Y$.
Consequently,
\begin{equation}\label{eq:gamma-bound}
\gamma(C)
  \le z_C
  \le rn_C
  \le rM_k.
\end{equation}

If $C$ is a connected component  with $\gamma(C)>0$,  then $|N_X(C)|>0$. Otherwise, $\gamma(C)=2e_C-rn_C\le0$, a contradiction. Moreover,
$1\le |N_X(C)|\le z_C\le rM_k=:R_k$. Let $\cH$ be a multi-hypergraph with vertex set $X$. For each connected component $C$ with $\gamma(C)>0$, we put $N_X(C)$ in the edge set $\cal E$ of $\cH$.
We claim that $\cH$ contains no Berge path of length $r+1$. Suppose  $v_0, N_X(C_1),v_1,\dots,N_X(C_{r+1}),v_{r+1}$ is a Berge path with length $r+1$.
For each $i\in\{1,\dots,r+1\}$, choose vertices
$c_i^-,c_i^+\in V(C_i)$ satisfying
\[
  v_{i-1}c_i^-\in E(G)
  \qquad\text{and}\qquad
  v_i c_i^+\in E(G).
\]
These vertices exist because $v_{i-1},v_i\in N_X(C_i)$. Since $C_i$ is
connected, it contains a simple $c_i^-c_i^+$-path, where we take the trivial one-vertex path if
$c_i^-=c_i^+$.
Then there is a  $v_{i-1}v_i$-path of length at least two whose internal vertices all lie
in $C_i$ for each $1 \leq i \leq r+1$. Thus we get  a simple path in $G$ of length at least  $2(r+1)=2r+2\ge k-1.$
This is a contradiction as we assume $G$ is $P_k$-free.

It remains to bound the multiplicity of $\cH$. For  $C \in \cC$, we call $G[V(C) \cup N_X(C)]$  the {\it type} of the pair $(C,N_X(C))$. Note that $|N_X(C)| \leq R_k$ as $d_X(y) \leq r$ for each $y \in Y$ and $|C| \leq M_k$ for each $C \in \cC$.
For a connected component $C$ with $m$  vertices, there are at most
$2^{\binom m2}$ choices for  edges in $C$ and
 at most $2^{m|N_X(C)|}\le2^{mR_k}$ choices for edges $E(C,X)$.
  Hence the
number of possible types  is at most   $ Q_k:=\sum_{m=1}^{M_k}2^{\binom m2+mR_k}.$

Fix a pair $(C,N_X(C))$  with $\gamma(C)>0$. We write $U=N_X(C)$ and
$s:=|N_X(C)|$. The number of pairs $(C',N_X(C'))$ such that $(C,N_X(C))$  and $(C',N_X(C'))$ are of the same type is at most $s$. Otherwise, there are  $s+1$ pairs $(C,U), (C_1,U), \ldots, (C_s,U)$ that are of the same type.
Now let $Z=G[C \cup U]$, $L=C$, $R=U$, $L_i=C_i$ for $1 \leq i \leq s$. Then we have the bouquet
$\mathcal B_{s+1}(Z,L,R)$. The assumptions  $d_Y(y)\le r$, $d_X(y)\le r$, and
$\gamma(C)>0$ ensure $\mathcal B_{s+1}(Z,L,R)$ satisfies conditions in Lemma  \ref{lem:bouquet}. By Lemma \ref{lem:bouquet},  $\mathcal B_{s+1}(Z,L,R)$ contains
$P_k$, a contradiction.

There are at most $Q_k$ types of pairs, and each type of pair occurs at most
$s=|N_X(C)|\le R_k$ times. Thus $\cH$ has
multiplicity at most $\mu_k:=R_kQ_k.$
Since $\cH$ has no Berge path of length $r+1$, Lemma \ref{lem:hypergraph}
gives
\begin{equation*}
  \bigl|\{C\in\cC:\gamma(C)>0\}\bigr|
  =e(\cH)
  \le D_k|X|,
  \qquad
  D_k:=r+\mu_k2^{r+1}.
\end{equation*}
Finally,  by using \eqref{eq:gamma-bound}, we obtain
\begin{align*}
  e_G(X,Y)+2e_G(Y)-r|Y|=\sum_{C\in\cC}\gamma(C)\le\sum_{\gamma(C)>0}\gamma(C)\le rM_kD_k|X|.
\end{align*}
Thus \eqref{eq:weighted-conclusion} holds with $C_k:=rM_kD_k.$
\end{proof}

\section{Proof of Theorem \ref{thm:main}}\label{sec:main-proof}
We first recall the following version of the triangle removal lemma.

\begin{lemma}[Ruzsa--Szemer\'edi~\cite{RS78}]\label{lem:triangle-removal}
For a fixed constant $\epsilon>0$, let $\delta=\delta(\epsilon)$ be a constant such that
$1/\delta$ is a tower of twos of height
$684\log\bigl(1/\epsilon\bigr)$. If $G$ is a graph on $n$ vertices with
at least $\epsilon n^2$ edge-disjoint triangles, then $G$ contains at
least $\delta n^3$ triangles.
\end{lemma}

The following lemma is an immediate consequence.

\begin{lemma}\label{lem:hitting}
Let $G$ be an $n$-vertex graph satisfying $t(G)=o(n^3)$. Then there is a
set $R_0\subseteq E(G)$ with $|R_0|=o(n^2)$ such that $G-R_0$ is
triangle-free.
\end{lemma}

\begin{proof}
Let $\cM$ be a maximal collection of pairwise edge-disjoint triangles in $G$. We first show that $|\cM|=o(n^2)$. Otherwise, there is a constant $\epsilon>0$ such that $|\cM|\ge\epsilon n^2$. Thus, the graph $G$ contains at least $\epsilon n^2$ edge-disjoint triangles. By Lemma \ref{lem:triangle-removal}, it  contains at least $\delta(\epsilon)n^3$ triangles. Since $\delta(\epsilon)>0$ is independent of $n$, this contradicts $t(G)=o(n^3)$. Therefore $|\cM|=o(n^2)$. Let
\[
  R_0:=\bigcup_{T\in\cM}E(T).
\]
The triangles in $\cM$ are pairwise edge-disjoint, so $|R_0|=3|\cM|=o(n^2)$. If $G-R_0$ contained a triangle $T$, then $T$ would be edge-disjoint from every triangle in $\cM$, and $\cM\cup\{T\}$ would be a larger collection of pairwise edge-disjoint triangles. This  contradicts the maximality of $\cM$. Hence $G-R_0$ is triangle-free and the lemma is proved.
\end{proof}

As  the lower bound construction can be found in  Mubayi and Mukherjee's paper \cite{MM23}, we omit  details here. It remains to prove the upper bound.

 For fixed $k\ge4$, let $G$ be an $n$-vertex $\widehat P_k$-free graph for sufficiently large $n$.
 Recall  $t(G)$ is the number of triangles in $G$ and $r=\lfloor \tfrac{k-2}{2}\rfloor$.
 We aim to prove $t(G)\le rn^2/8+o(n^2)$.

 \begin{claim}
 $t(G)=o(n^3)$.
 \end{claim}
\begin{proof}
For every $v\in V(G)$, the subgraph $G[N_G(v)]$ is $P_k$-free as $G$ is $\widehat P_k$-free. Theorem \ref{lem:EG} implies
$ e(G[N_G(v)])\le\tfrac{k-2}{2}d_G(v).$ Summing over $v$,
we obtain
\begin{equation*}
  3t(G)
  =\sum_{v\in V(G)}e(G[N_G(v)])
  \le\frac{k-2}{2}\sum_{v\in V(G)}d_G(v)
  =(k-2)e(G).
\end{equation*}
It follows that
\[
  t(G)\le\frac{k-2}{3}\binom n2=O_k(n^2)=o(n^3).
\]
\end{proof}

By Lemma \ref{lem:hitting}, there is a set $R_0\subseteq E(G)$ with
$|R_0|=o(n^2)$ such that $G-R_0$ is triangle-free. If we put
\[
  \rho_n:=\frac{|R_0|}{n^2},
\]
then $\rho_n\to0$.  If $\rho_n=0$, then we take $\omega_n:=n^{1/4}$. When $\rho_n>0$, we set $\omega_n:=\min\{n^{1/4},\rho_n^{-1/2}\}$. Note that in each case,
\begin{equation}\label{eq:omega-properties}
  \omega_n \to\infty
  \qquad\text{and}\qquad
  \omega_n \rho_n\to 0  \qquad\text{as}\qquad n \to \infty.
\end{equation}
Let
\[
  R_0':=\{uv\in E(G)\setminus R_0:\codeg_G(uv)>\omega_n\} \qquad\text{and}\qquad R=R_0 \cup R_0'.
\]
\begin{claim}
$|R|=o(n^2)$.
\end{claim}
\begin{proof} Recall that $|R_0|=o(n^2)$.
Observe   $ \omega_n |R_0'| \le \sum_{uv \in R_0'}\codeg_G(uv) \le3t(G)=O_k(n^2)$.
As $\omega_n \to\infty$, we get $|R_0'|=o(n^2)$ and then $|R|=o(n^2)$.
\end{proof}
Let $T(G)$ be the set of triangles in $G$. We define
\[
T_1(G)=\{T \in T(G): |E(T) \cap R| \geq 2\} \qquad\text{and}\qquad T_2(G)=T(G)\setminus T_1(G).
\]
\begin{claim}\label{t1g}
$|T_1(G)|=o(n^2)$.
\end{claim}
\begin{proof}
For each $v \in V(G)$, we define $N_R(v)=\{u \in V(G): uv \in R\}$. Notice that the subgraph of $G$ induced by $N_R(v)$ contains at most $\tfrac{k-2}{2}|N_R(v)|$ edges by Theorem \ref{lem:EG}.

Thus, for each $v\in V(G)$, there are at most
$\tfrac{k-2}{2}|N_R(v)|$ triangles in which both edges incident with $v$ belong to  $R$. Every triangle in $\mathcal T_1(G)$ is counted for at least one vertex $v$. Summing over all vertices,  we obtain
\[
|T_1(G)| \leq \sum_{v \in V(G)} \frac{k-2}{2}|N_R(v)|=(k-2)|R|=o(n^2).
\]
\end{proof}
We next aim to estimate $|T_2(G)|$. Since $G-R_0$ is triangle-free, each triangle must contain an edge in $R_0$.
 The definition of  $T_2(G)$ gives  the following fact.
\begin{fact}
Each triangle from $T_2(G)$ contains exactly one edge in $R_0$ and no edges in $R_0'$.
\end{fact}
Let $H=E(G) \setminus R$. Note that the subgraph of $G$ spanned by the edge set $H$  is triangle-free and every edge of $H$ lies in at most
$\omega_n$ triangles in $G$. In addition, each triangle of $T_2(G)$ contains exactly two edges in $H$ and one edge in $R_0$.

Recall that $R=R_0\cup R_0'$. Let $c\colon E(G)\setminus R_0'\to\{0,1\}$ be a coloring of $E(G)\setminus R_0'$ such that
\[
  c(e)=
  \begin{cases}
    1,&e\in H,\\
    0,&e\in R_0.
  \end{cases}
\]
Define a graph $\widetilde G$ with vertex set $V(G)\times\{0,1\}$.   For $(x,i), (y,j) \in V(\widetilde G)$,
\[
(x,i) \text{ is adjacent to } (y,j) \text{ if and only if } xy \in E(G)\setminus R_0' \text{ and } i-j \equiv c(xy) \pmod{2}.
\]
Notice that the above definition actually introduces a surjection $f: E(\widetilde G) \to E(G)\setminus R_0' $. For each $xy \in E(G) \setminus R_0'$, we use $\widetilde{xy}$ to denote the pre-image of $xy$ and then $|\widetilde{xy}|=2$.
Let $\widetilde {H}$ and $\widetilde R_0$ be the pre-images of $H$  and
$R_0$, respectively.  Note that $|\widetilde H|=2|H|$ and $|\widetilde R_0|=2|R_0|=o(n^2).$

Let $A:=V(G)\times\{0\}$ and  $B:=V(G)\times\{1\}$. Observe that each edge in $\widetilde H$ is crossing. Meanwhile, each edge in $\widetilde R_0$ is  contained in either $A$ or $B$. Let $\widetilde{T}$ be the set of triangles in $\widetilde{G}$ that consists of two edges from $\widetilde H$ and one edge from $\widetilde R_0$.
\begin{claim}\label{2t2}
$|\widetilde{T}|=2|T_2(G)|.$
\end{claim}
\begin{proof}
Suppose $uvw \in T_2(G)$ with $uv,vw \in H$ and $uw \in R_0$. We next show the pre-image of $uvw$ contains two triangles in $\widetilde{T}$. By the definition of edges in $E(\widetilde G)$, we have $\widetilde{uv}=\{(u,0)(v,1)$, $(u,1)(v,0)\}$ and each possibility gives us a unique choice for $(w,i)$ with $i \in \{0,1\}$. The claim is proved.
\end{proof}

\begin{claim}\label{lem:cover-properties}
The graph $\widetilde G$ is $\widehat P_k$-free. Moreover, every edge of
$\widetilde H$ is contained in at most $\omega_n$ triangles in
$\widetilde T$.
\end{claim}
\begin{proof}
We define a mapping $\pi: V(\widetilde G) \to V(G)$ such that  $\pi((v,i))=v$ for each $(v,i) \in V(\widetilde G)$.  If
 $\widetilde G$ is not  $\widehat P_k$-free, then there is a vertex $(v,i)$ such that $N_{\widetilde G}((v,i))$ contains a $P_k$. Note that $\pi$ is injective on  $N_{\widetilde G}((v,i))$.
  If $(v,i)(u,j) \in E(\widetilde G)$, then $\pi((v,i))\pi((u,j)) \in E(G)$ by the definition of $E(\widetilde G)$. It follows that
$\pi(N_{\widetilde G}((v,i)))$ contains a $P_k$ in $G$. Apparently, each vertex from $\pi(N_{\widetilde G}((v,i)))$ is adjacent to $v$ in $G$. Thus $G$ contains a  $\widehat P_k$, a contradiction. The first part of the claim follows.

 For the second part of the claim, let $(v,0)(u,1) \in \widetilde H$ with $(v,0) \in A$ and $(u,1) \in B$.
Each triangle in $\widetilde T$ containing $(v,0)(u,1)$ has the form  $(v,0)(u,1)(w,i)$, where $w \in N_G(u)\cap N_G(v)$ and $i \in \{0,1\}$. For each $w \in N_G(u)\cap N_G(v)$, we next show that  at most one of  $(v,0)(u,1)(w,0)$ and $(v,0)(u,1)(w,1)$ is a  triangle in $\widetilde G$. Assume $(v,0)(u,1)(w,0)$ is a triangle in $\widetilde G$. Then $vw \in  R_0$, so $(v,0)$ is not adjacent to $(w,1)$ in $\widetilde G$ and  $(v,0)(u,1)(w,1)$ does not form a triangle. Similarly, we can show $(v,0)(u,1)(w,0)$ does not form a triangle provided $(v,0)(u,1)(w,1)$ is a  triangle. As $\codeg_G(u,v) \leq \omega_n$, the second part of the claim is proved.
\end{proof}

We now turn to estimating $\widetilde T$. All edges of $\widetilde R_0$ lie in the parts $A$ and $B$. Meanwhile, all edges of $\widetilde H$ are crossing $A$ and $B$.
For an edge $ab \in \widetilde H$ with $a \in A$ and $b \in B$, we say $ab$ is {\it bad} if either $|N_A(a) \cap N_A(b)| \geq r+1$ or $ |N_B(a) \cap N_B(b)| \geq r+1$.
The set of bad edges in $\widetilde H$ is denoted by $\widetilde{E}_{\textrm{bad}}$.
Let
\[
\widetilde T_1=\{T \in \widetilde{T}: E(T) \cap \widetilde{E}_{\textrm{bad}} \neq \emptyset\}, \qquad \widetilde T_2=\widetilde T \setminus \widetilde T_1.
\]
\begin{claim}\label{wt1}
$|\widetilde T_1|=o(n^2).$
\end{claim}
\begin{proof}
 Fix $a \in A$ and  let $Q_a$ be the bipartite graph with parts  $N_A(a)$ and $N_B(a)$. Apparently,  $Q_a$ is $P_k$-free. Define
\[
  N_B^+(a):=\{b \in N_B(a): |N_A(a)\cap N_A(b)|\ge r+1\}.
\]
Applying Theorem \ref{lem:EG} to the $P_k$-free graph
$Q_a[N_B^+(a) \cup N_A(a)]$, we obtain
\begin{equation}\label{eq:Yb-plus-bound-start}
  (r+1)|N_B^+(a)|
  \le e_{Q_a}(N_B^+(a),N_A(a))
  \le\frac{k-2}{2}\bigl(|N_B^+(a)|+|N_A(a)|\bigr).
\end{equation}
If $k=2r+2$, then $(k-2)/2=r$, and
\eqref{eq:Yb-plus-bound-start} gives
\[
  (r+1)|N_B^+(a)|
  \le r\bigl(|N_B^+(a)|+|N_A(a)|\bigr),
\]
so $|N_B^+(a)|\le r|N_A(a)|$. If $k=2r+3$, then
$(k-2)/2=(2r+1)/2$, and after multiplying by $2$ we obtain
\[
  2(r+1)|N_B^+(a)|
  \le(2r+1)\bigl(|N_B^+(a)|+|N_A(a)|\bigr),
\]
so $|N_B^+(a)|\le(2r+1)|N_A(a)|$. Thus in both cases, we get
\begin{equation}\label{eq:Yb-plus-bound}
  |N_B^+(a)|\le(2r+1)|N_A(a)|.
\end{equation}
Summing \eqref{eq:Yb-plus-bound} over $a\in A$ gives
\[
  |\{ab\in {\widetilde E}_{\textrm{bad}}: |N_A(a) \cap N_A(b)|\ge r+1\}|
  =\sum_{a \in A}|N_B^+(a)|
  \le(2r+1)\sum_{a\in A}|N_A(a)|
  =2(2r+1)e(A).
\]
Repeating the argument above, we can show
\[
  |\{ab\in {\widetilde E}_{\textrm{bad}}: |N_B(a) \cap N_B(b)|\ge r+1\}| \leq 2(2r+1)e(B).
\]
Therefore,
\[
|{\widetilde E}_{\textrm{bad}}| \leq  2(2r+1)(e(A)+e(B))=2(2r+1)|\widetilde{R}_0|=O_k(|\widetilde{R}_0|).
\]
By Claim 6, each edge $ab\in {\widetilde E}_{\textrm{bad}}$ is contained in at most $\omega_n$ triangles of $\widetilde{T}_1$. This implies that
\begin{equation*}
  |\widetilde{T}_1| \leq \omega_n |{\widetilde E}_{\textrm{bad}}|
  =O_k(\omega_n|R_0|)
  =O_k(\omega_n\rho_n n^2)
  =o(n^2),
\end{equation*}
where the final step follows from \eqref{eq:omega-properties}. The claim is proved.
\end{proof}
The following claim will complete the proof of the upper bound   for Theorem \ref{thm:main}.
%Recall that the definition of edge set $\widetilde H$, we write $|H|$ for the number of elements in $H$.
\begin{claim}\label{wt2}
$|\widetilde T_2| \leq \tfrac{r}{2} |\widetilde H|+o(n^2).$
\end{claim}
\begin{proof}
Let $\widetilde{E}_{\textrm{good}}=\widetilde H \setminus \widetilde{E}_{\textrm{bad}}$. Now each triangle from $\widetilde T_2$ contains precisely two edges from $\widetilde{E}_{\textrm{good}}$ and one edge from $\widetilde R_0$. We write
 $\widetilde T_2= \widetilde T_2' \cup  \widetilde T_2''$, where
\[
\widetilde T_2'=\{aba' \in \widetilde T_2: ab, a'b \in \widetilde{E}_{\textrm{good}}, aa' \in \widetilde R_0, \text{ and }a,a'\in A\}
\]
and
\[
\widetilde T_2''=\{abb' \in \widetilde T_2: ab, ab' \in \widetilde{E}_{\textrm{good}}, bb' \in \widetilde R_0,\text{ and }b,b'\in B\}.
\]
 For a vertex $a \in A$, let $N_B'(a)=\{b \in B: ab \in \widetilde{E}_{\textrm{good}}\}$ and $Q_a$ be the  graph with parts $N_A(a)$ and $N_B'(a)$. The edge set of $Q_a$ consists of edges inside $N_B'(a)$ as well as those good edges between $N_A(a)$ and $N_B'(a)$. Each vertex $b \in N_B'(a)$ has at most $r$ neighbors in $N_A(a)$ as $ab$ is good. Similarly, $b$ is adjacent to at most $r$ vertices in $N_B'(a)$. It is clear that $Q_a$ is $P_k$-free.  Applying Theorem \ref{lem:weighted} with $X=N_A(a)$ and $Y=N_B'(a)$, we conclude that
\begin{equation*}
  e_{Q_a}(N_A(a),N_B'(a))+2e_{Q_a}(N_B'(a))
  \le r|N_B'(a)|+C_k|N_A(a)|.
\end{equation*}
Observe that
$\sum_{a \in A}e_{Q_a}(N_A(a),N_B'(a))=2|\widetilde T_2'|$ and $\sum_{a \in A} e_{Q_a}(N_B'(a))=|\widetilde T_2''|$.
Therefore,
\[
\sum_{a \in A}(e_{Q_a}(N_A(a),N_B'(a))+2e_{Q_a}(N_B'(a)))=2|\widetilde T_2|.
\]
We also have
\begin{align*}
\sum_{a \in A}(e_{Q_a}(N_A(a),N_B'(a))+2e_{Q_a}(N_B'(a))) &\leq \sum_{a \in A} (r|N_B'(a)|+C_k|N_A(a)|) \\
                                                         & \leq r |\widetilde{E}_{\textrm{good}}|+C_k|\widetilde R_0| \\
                                                         & \leq r |\widetilde{H}|+o(n^2).
\end{align*}
The claim follows from the inequalities above.
\end{proof}

\begin{proof}[Proof of Theorem \ref{thm:main}] By Claims \ref{t1g} and \ref{2t2},
\[
t(G)=|T_1(G)|+|T_2(G)|=|T_2(G)|+o(n^2)=\frac{1}{2}|\widetilde T|+o(n^2).
\]
Further, Claims \ref{wt1} and \ref{wt2} imply
\[
t(G)=\frac{1}{2}|\widetilde T_1|+\frac{1}{2}|\widetilde T_2|+o(n^2)=\frac{1}{2}|\widetilde T_2|+o(n^2) \leq \frac{r}{4}|\widetilde H|+o(n^2)=\frac{r}{2}|H|+o(n^2).
\]
Since the subgraph spanned by the edge set $H$ is triangle-free,  Mantel's theorem gives
$|H| \le n^2/4$. Therefore, we obtain
\begin{equation*}\label{eq:upper-final}
  t(G)
  \le\frac r2 |H|+o(n^2)
  \le\frac r2\cdot\frac{n^2}{4}+o(n^2)
  =\frac{rn^2}{8}+o(n^2),
\end{equation*}
as desired. The proof of Theorem \ref{thm:main} is complete.
\end{proof}

\noindent
{\bf Acknowledgement:} Generative AI tools were used to explore proof strategies for Theorem \ref{lem:weighted} and
all other arguments and proofs in this paper were developed by the authors.  Generative AI tools were also used for  proofreading. All AI-assisted outputs were
independently checked and verified by the authors, who take full responsibility for the content of this work.

\end{document}